\documentclass{amsart}

\usepackage{amsfonts,amssymb,amsmath,amsthm}
\usepackage{url}
\usepackage{enumerate}

\usepackage{graphics}
\usepackage{graphicx}

\newtheorem{theorem}{Theorem}

\newtheorem{definition}{Definition}	
\newtheorem{remark}{Remark}
\newtheorem{cor}{Corollary}
\newtheorem{prop}{Proposition}

\newtheorem{example}{Example}

\def\K{{\mathcal{ K }}}
\def\k{{\kappa}}
\def\l{{\lambda}}

\def\age{{\rm{ Age }}}

\def\aut{{\rm{ Aut }}}
\def\lo{{\rm{ LO }}}

\def\homeo{{\rm{ Homeo }}}

\title{More Universal Minimal Flows of Groups of Automorphisms of Uncountable Structures}
\author{Dana Barto\v{s}ov\'a}

\address{Department of Mathematics\\
University of Toronto\\
Bahen Center\\
40 St. George St.\\
Toronto\\
Ontario\\
Canada\\
M5S 2E4}
\email{dana.bartosova@utoronto.ca}
\address{Faculty of Mathematics and Physics\\
Charles University in Prague\\
Ke Karlovu 3\\
121 16 Prague\\
Czech Republic}

\thanks{This project was partially supported by the grant GAUK 66509}


\begin{document}
\maketitle

\textbf{Abstract.} In this paper, we compute universal minimal flows of groups of automorphisms of uncountable $\omega$-homogeneous graphs, $K_n$-free graphs, hypergraphs, partially ordered sets, and their extensions with an $\omega$-homogeneous ordering. We present an easy construction of such structures, expanding the jungle of extremely amenable groups.

\section{Introduction}
This paper is a continuation of \cite{DB}, where we used the ultrafilter flow to compute universal minimal flows of automorphism groups of uncountable  structures acting transitively on isomorphic finite substructures and such that the class of finite substructures satisfies the Ramsey property and admits an appropriate extension to a class of linearly ordered structures (e.g. Boolean algebras, vector spaces over a finite field, linearly ordered sets etc.). Here, we apply this result for linearly ordered structures to generalize techniques of Kechris, Pestov and Todorcevic in \cite{KPT} to compute universal minimal flows of more automorphism groups of uncountable structures (e.g. (linearly ordered) graphs, hypegraphs, posets etc.).


In the first section, we remind the  reader of the basic notions from topological dynamics and groups of automorphisms of structures. We recall a theorem characterizing extremely amenable groups of automorphisms  and give new examples. If $G$ is a group of automorphisms of a structure $A,$ we can assume that $A$ is $\omega$-homogeneous, i.e. every partial isomorphism of finitely generated substructures of $A$ can be extended to a full isomorphism of $A$ (see \cite{H}).
 As in \cite{KPT}, we consider the action of  $G$  on the compact space $\lo(A)$ of all linear orderings on $A.$ For a linear ordering $<\in\lo(A),$ we denote by $\overline{G<}$ the topological closure of the orbit $G<=\{g<:g\in G\}$ of $<$ under the action of $G$ in the space $\lo(A).$ If there exists a linear ordering $<\in\lo(A)$ making $(A,<)$ $\omega$-homogeneous, then as in Theorem 7.4 of \cite{KPT} minimality of the space $\overline{G<}$ corresponds to the ordering property for the class of finite substructure of $(A,<)$ that has been isolated in \cite{KPT}. In accordance with Theorem 7.5 in \cite{KPT}, universality of $\overline{G<}$ is then implied by the class of finite substructures of $(A,<)$ satisfying the Ramsey property.

In the next two sections, we show that there are many structures in every uncountable cardinality for which we can compute the universal minimal flow of their groups of automorphisms.
 
In the second section, we recall J\'osson's construction of universal homogeneous structures of cardinalitites $\kappa$ whenever $\kappa^{<\kappa}=\kappa.$ We show  that results of \cite{KPT} can be extended from Fra\"iss\'e structures to J\'onsson structures without any obstacles and compute universal minimal flows of groups of automorphisms of universal homogeneous graphs, $K_n$-free graps, hypergraphs, $\mathcal{A}$-free hypergraphs and posets of the relevant cardinality. 

In the last section, we fill in the gap given by the restriction in J\'onsson's construction on cardinality. Given a graph ($K_n$-free graph, hypergraph, $\mathcal{A}$-free hypergraph, poset) of an arbitrary cardinality, we find a superstructure of the same type and the same cardinality that is $\omega$-homogeneous with an $\omega$-homogeneous linear ordering to which  we apply results from the first section to compute the universal minimal flows of their automorphism groups.


\section{General theory}

Let $G$ be a topological group with identity element $e$ and $X$ a compact Hausdorff space. We say that $X$ is a \emph{$G$-flow} via an action $\pi,$ if $\pi:G\times X\to X$ is a continuous map satisfying $\pi(e,x)=x$ for every $x\in X$ and $\pi(gh,x)=\pi(g,\pi(h,x))$ for every $g,h\in G$ and $x\in X.$ Usually, $\pi$ is understood and we simply say that $X$ is a $G$-flow and write $gx$ in place of $\pi(g,x).$ A $G$-flow $X$ is called \emph{minimal} if there is no closed subspace of $X$ invariant under the ation of $G.$ Equivalently, $X$ is minimal if for every $x\in X$ the \emph{orbit} $Gx=\{gx:g\in G\}$ is dense in $X.$

In what follows, we will be interested in automorphism groups of structures in a countable signature and their dense subgroups. If $A$ is a structure, we  denote its group of automorphisms by $\aut(A).$ We consider $A$ as a discrete space and equip $\aut(A)$ with the topology of pointwise convergence  turning it into a topological group. If $G$ is a subgroup of $\aut(A),$ then the topology on $G$ is generated by pointwise stabilizers of finite substructures of $A,$ where the \emph{pointwise stabilizer} $G_F$ of a substructure $F$ of $A$ is the following clopen subgroup:

$$
G_F=\{g\in G:g(a)=a \ \forall \ a\in F\}.
$$

For a cardinal $\k,$ denote by $S_{\k}$ the group of all bijections on $\k$ with the topology of pointwise convergence. If $A$ is a structure of cardinality $\kappa,$ then $\aut(A)$ is naturally isomorphic to a closed subgroup of $S_{\k}$ (see e.g. \cite{H}).

We say that $A$ is \emph{$\omega$-homogeneous}, if every partial isomorphism between two finitely generated substructures of $A$ can be extended to an automorphism of $A.$

We will identify universal minimal flows of groups of automorphisms of certain structures with spaces of linear orderings. Let us denote by $\lo(\kappa)$ the space of all linear orderings on $\k$ with the topology inherited from the Tychonoff product $2^{\kappa\times\kappa}.$ $\lo(\k)$ is then a compact space and $S_{\k}$ naturally acts on $\lo(\k)$ as follows:

$$
\sigma\in S_{\k}, <\in\lo(\k),\alpha,\beta < \kappa \rightarrow (\alpha (\sigma <) \beta \leftrightarrow \sigma^{-1}(\alpha)<\sigma^{-1} \beta). 
$$

Whenever we talk about an action of a group $G\leq \aut(A),$ where $A$ is a structure of cardinality $\k,$ on the space of linear orderings, we always refer to the above action with $\aut(A)$ identified with a closed subgroup of $S_{\k}.$

It follows from \cite{DB} that whenever $A$ is a structure and $G$ is a dense subgroup of $\aut(A),$ then the universal minimal flow of $G$ is the universal minimal flow of $\aut(A)$ under the restricted action.

We write $B\leq A$ to denote that $B$ is a substructure of $A$ and $\age(A)$ denotes the class of finitely generated substructures of $A$ up to an isomorphism.

In computations of universal minimal flows of groups of automorphisms, the Ramsey property for finite structures turns out to play a crucial role.

\begin{definition}[Ramsey property]
 A class $\mathcal{K}$ of finite structures satisfies the \emph{Ramsey property} if for every $A\leq B\in\mathcal{K}$ and $k\geq 2$ a natural number there exists $C\in\mathcal{K}$ such that 
 $$
 C\to (B)^A_k.
 $$
 It means for every colouring of copies of $A$ in $C$ by $k$ colours, there is a copy $B'$ of $B$ in $C,$ such that all copies of $A$ in $B'$ have the same colour.
\end{definition}

\begin{example} \label{E1}
The following classes of finite structures satisfy the Ramsey property:
\begin{itemize}
\item finite graphs equipped with arbitrary linear orderings,
\item finite $K_n$-free graphs with arbitrary linear orderings for some $n\in \omega,$
\item finite hypergraps of a given signature with arbitrary linear orderings,
\item finite $\mathcal{A}$-free hypergraphs of a given signature with arbitrary linear orderings, 
\item finite posets with linear orderings extending the partial order. 
\end{itemize}

For graph and hypergraph classes see \cite{NR1} and \cite {NR2}, for posets see \cite{MS}.
\end{example}

By a \emph{graph} we mean an unordered graph. A \emph{$K_n$-free graph} is a graph that does not contain the complete graph on $n$-vertices as an iduced subgraph. A \emph{hypergraph} $H$ is a structure in a finite signature $L=\{R_i:i<k\}$ of relational symbols, where each $R_i$ is closed under permutations, i.e. if $R_i$ has arity $l$ and $\sigma$ is a permutation on $\{0,1,\ldots,l-1\},$ then $(h_0,h_1,\ldots,h_{l-1})\in R_i^H\subset H^l$ implies that all $h_0,h_2,\ldots, h_{l-1}$ are distinct and $(h_{\sigma(0)},h_{\sigma(1)},\ldots, h_{\sigma(l-1)})\in R_{i}^H.$ In other words, we can think of $R_i^H$ as a collection of subsets of $H$ of size $l.$ A hypergraph $H$ is called \emph{irreducible} if it contains at least two elements and whenever $x\neq y$ in $H$ then there exists $i<k$  and $S\in R_i^H$ such that $\{x,y\}\subset S.$ Let $\mathcal{A}$ be a class of irreducible hypergraphs in signature $L.$ A hypergraph is \emph{$\mathcal{A}$-free} if no element of $\mathcal{A}$ can be embedded into it. By a \emph{poset}, we mean a partially ordered set. 

We are ready to recall a characterization of extremely amenable subgroups of $S_{\k},$ that is to say subgroups whose universal minimal flow is trivial.

\begin{definition}[Extreme amenability]
A topological group $G$ is called \emph{extremely amenable} if the universal minimal flow of $G$ is a single point.
\end{definition}

The following theorem was proved in \cite{KPT} to characterize extremely amenable groups of automorphisms of countable structures. It was varified in \cite{DB} that it works for uncountable structures as well.

\begin{theorem} \label{T2}
Let $G$ be a subgroup of $S_{\kappa}.$ The following are equivalent:
\begin{itemize}
\item[(a)] $G$ is extremely amenable.
\item[(b)] $(i)$ for every finite $F\subset\kappa,$ $G_F=G_{(F)}$, where $G_{(F)}=\{g\in G: gF=F\}$ and $(ii)$ for every colouring $c:G/G_F\to\{1,2,\ldots, k\}$ and  for every finite $B\supset F,$ there is $g\in G$ and $i\in\{1,2,\ldots,n\}$ such that $c(hG_F)=i$ whenever $h[F]\subset g[B].$
\item[(c)] $(i')$ $G$ preserves an ordering and $(ii)$ as above.
\end{itemize}
\end{theorem}

\begin{remark}
Let $A$ be an $\omega$-homogeneous relational structure such that $G$ is dense in its automorphism group. Since finitely generated substructures of $A$ are finite, $(ii)$ of $(b)$ simply says that $\age(A)$ satisfies the Ramsey property. Whence the following corollary.
\end{remark}

\begin{cor}
Let $A$ be an $\omega$-homogeneous linearly ordered structure with finitely generated substructures being finite and let $G$ be a dense subgroup of $\aut(A).$ Then $G$ is extremely amenable if and only if $\age(A)$ satisfies the Ramsey property.
\end{cor}

Since the classes of finite structures in Example \ref{E1} are all relational and satisfy the Ramsey property, we get new examples of extremely amenable groups of uncountable structures (the result for countable structures was proved in \cite{KPT}).

\begin{cor}
Let $A$ be an uncountable graph ($K_n$-free graph, hypergraph, $\mathcal{A}$-free hypergraph, poset) and let $<$ be a linear ordering such that $(A,<)$ is $\omega$-homogeneous (and if $A$ is a poset, $<$ extends the partial order). Let $G$ be a dense subgroup of $\aut(A,<).$ If $\age(A,<)$ is the class of all finite  linearly ordered graphs ($K_n$-free graph, hypergraph, $\mathcal{A}$-free hypergraph, posets with the linear ordering extending the partial order), then $G$ is extremely amenable.
\end{cor}

We are to compute universal minimal flows of groups of automorphisms of $\omega$-homogeneous structures that are not linearly ordered, but admit a suitable linearly ordered $\omega$-homogeneous extension. In \cite{KPT}, it was shown that ``suitable'' can be expressed in terms of the ordering property. 

\begin{definition}[Ordering property]
Let $L\supset \{>\}$ be a signature and let $\K_{<}$ be a class of $L$-structures where $<$ is interpreted as a linear ordereding. Let $L_0=L\setminus \{<\}$ and $\K=\K_{<}|L_0.$ We say that $\K_<$ has the \emph{ordering property} if for every $A\in \K$ there exists $B\in \K$ such that for every linear ordering $\prec$ on $A$ and for every linear ordering $\prec'$ on $B,$ if $(A,\prec)\in\K_<$ and $(B,\prec')\in\K_<$ then $(A,\prec)\leq (B,\prec').$ 
\end{definition}

All classes of structures in Example \ref{E1} satisfying the Ramsey property also satisfy the ordering property.

The following three theorems are Theorems 7.1., 7.4. and 7.5.  from \cite{KPT}. Their proofs in \cite{KPT} are given for countable structures, but with slight modifications they work for uncountable structures as well, since the topology on an  automorphism group of arbitrary size is given by finitely generated  substructures. 

\begin{theorem}\label{Torder}
Let $A$ be a structure, $\prec$ a linear ordering on $A$ and let $G$ be a dense subgroup of $\aut(A).$ Then $<\in \overline{G\prec}$ if and only if for every $B\in\age(A),$ $(B,<|B)\in\age(A,\prec).$ 
\end{theorem}

\begin{proof}
Let $<\in \overline{G\prec}.$ It means that for every $B\in\age(A)$ there exists $g\in G$ such that $g\prec|B=<|B.$ Thus $g:g^{-1}(B)\to B$ is an isomorphism between $(g^{-1}(B),\prec|g^{-1}(B))$ and $(B,<|B),$ which shows that $(B,<|B)\in\age(A,\prec).$

Conversely, suppose that $<\in\lo(A)$ and $(B,<|B)\in\age(A,\prec)$ for every $B\in\age(A).$ This means that for every $B\in\age(A)$ there is an embedding $i:(B,<|B)\to (A,\prec).$ If $C=i(B),$ then $i^{-1}$ is an isomorphism between $(C,\prec|C)$ and $(B,<|B)$ and in particular $i^{-1}$ is an isomorphism between $C$ and $B.$ By $\omega$-homogeneity of $A,$ $i^{-1}$ can be extended to a $g\in G.$ Then $<|B=g\prec |B$ and thus  $<\in \overline{G\prec}.$
\end{proof}

\begin{theorem}
Let $A$ be an $\omega$-homogeneous structure and let $\prec$ be a linear ordering on $A$ such that $(A,\prec)$ is $\omega$-homogeneous as well. Let $G$ be a dense subgroup of $\aut(A).$  Then the following are equivalent:
\begin{itemize}
\item[$(a)$] $\age(A,\prec)$ satisfies the ordering property,
\item[$(b)$] $\overline{G\prec}$ is a minimal $G$-flow.
\end{itemize} 
\end{theorem}

\begin{proof}
$(a) \Rightarrow (b)$ Let $<\in\overline{G\prec}.$ We would like to show that also $\prec\in\overline{G <}.$ By the previous theorem, this is equivalent to showing that $\age(A,\prec)\subset \age(A,<).$



To that end, let $(B,\prec|B)\in\age(A,\prec)$ and find $D\in\age(A)$ given by the ordering property for $B.$ Let $D_{0}$ be an isomorphic copy of $D$ in $A$ and let $i:(B,<|B)\to (D_0,\prec|D_0)$ be an embedding ensured by the ordering property. Then $i$ is an isomorphism between $(B,<|B)$ and $(i(B),\prec|i(B))\in\age(A,\prec),$ showing $(B,<|B)\in\age(A,\prec).$

$(b)\Rightarrow (a)$
Given $B\in\age(A)$ we need to find $D\in\age(A)$ such that whenever $(B,\prec')\in\age(A,\prec)$ and  $(D,<)\in\age(A,\prec),$ there is an embedding of $(B,\prec')$ into $(D,<).$ Fix $(B,\prec')\in\age(A,\prec).$ For every $C\in\age(A)$ consider the set
$$
X_C=\{<\in\overline{G\prec} : (B,\prec'|B) \cong (C,<|C)\}.
$$ 
Then $\overline{G\prec}=\bigcup_{C\in\age(A)}X_C.$ Since each $X_C$ is open, there are $C_1,C_2,\ldots, C_n\in\age(A)$ such that $\overline{G\prec}=\bigcup_{i=1}^n X_{C_i}$ by compactness of $\overline{G\prec}.$ Let $D_{\prec'}$ be the substructure of $A$ generated by $\bigcup_{i=1}^n C_i.$ We need to show that whenever $(D_{\prec'},<)\in\age(A,\prec),$ then $(B,\prec')\leq (D_{\prec'},<).$ It means we need to find $<'\in\overline{G\prec}$ extending $<.$ By minimality of $\overline{G\prec},$ we know that there exists an embedding $i:(D_{\prec'},<) \to (A,\prec).$ By $\omega$-homogeneity of $A,$ there is $g\in G$ extending $i$ to all of $A.$ Then we get that $g^{-1} \prec|D_{\prec'}=<,$ so $g^{-1}\prec\in G\prec$ is the sought for extension of $<.$ 

Now we repeat the procedure for every linear ordering $\prec'$ on $B$ with $(B,\prec')\in\age(A,\prec)$ and set $D$ to be the substructure of $A$ generated by 
$$\bigcup_{(B,\prec')\in \age(A,\prec)} D_{\prec'}.$$
Then $D$ is a witness of the ordering property for $B$ and we are done.  
 
\end{proof}

\begin{theorem}\label{TUM}
Let $A$ be an $\omega$-ultrahomogeneous structures and let $\prec$ be a linear ordering on $A$ such that $(A,\prec)$ is also $\omega$-homogeneous. Suppose that finitely generated substructures of $A$ are finite. Let $H=\aut(A)$ and $H_{\prec}=\aut(A,\prec).$ 
\begin{itemize}
\item[$(a)$] Suppose that $\age(A,\prec)$ satisfies the Ramsey property. Then $(\overline{H\prec},\prec)$ is the universal ambit among those $H$-ambits whose base point is fixed by the action of $H_{\prec}.$
\item[$(b)$] Suppose that $\age(A,\prec)$ satisfies both the Ramsey and the ordering properties. Then $\overline{H\prec}$ is the universal minimal $H$-flow.
\end{itemize}
\end{theorem}

\begin{proof}
$(a)$ Let $(X,x_0)$ be an $H$-flow such that $H_{\prec}x_0=\{x_0\}.$ Let $\Phi$ be the closure of the following set in the compact Hausdorff space $\overline{H\prec}\times X.$
$$
\{(h\prec,hx_0):h\in H\}\subset \overline{H\prec}\times X.
$$
We will show that $\Phi$ is a graph of a function $\phi:\overline{H\prec} \to X.$ Having proved this, it is easy to verify that $\phi$ works:  Since $\Phi$ is closed, $\phi$ is continuous. Also, $\phi$ is an $H$-homomorphism: If $(<,x)\in\Phi,$ then there is a net $\{h_i:i\in I\}$ such that $\{h_i\prec:i\in I\}$ converges to $<$ and $\{h_i x_0:h_i\in I\}$ converges to $x.$ Let $h\in H.$ Then $(hh_i\prec,hh_ix_0)\in\Phi$ for every $i\in I,$ and $\{(hh_i,hh_ix_0):i\in I\}$ converges to $(h<,hx)\in\Phi.$ It follows that $\phi(h<)=hx=h\phi(\prec).$ Finally, since $Hx_0$ is dense in $X,$ $\phi$ is surjective.

First, let us show that for every $<\in \overline{H\prec}$ there is $x\in X$ such that $(<,x)\in\Phi.$ Indeed, let $\{h_i:i\in I\}$ be a net such that $\{h_i \prec\}$ converges to $<.$ Then $\{h_i x_0\}$ is a net in $X,$ so by compactness of $X$ there is a subnet $h_i^j x_0$ converging to some $x\in X.$ Then $(h_i^j\prec,h_i^j x_0)$ converges to $(<,x)\in \Phi.$

Second, we prove that if $(<,x_1),(<,x_2)\in \Phi,$ then $x_1=x_2.$ Let $\{h_i:i\in I\}$ and $\{g_j:j\in J\}$ be nets such that $(h_i\prec,h_ix_0)$ is a net converging to $(<,x_1)$ and $(g_j\prec, g_jx_0)$ is a net converging to $(<,x_2)$ and suppose that $x_1\neq x_2.$ 

As $X$ is compact Hausdorff, it is regular, so there are open neighbourhoods $U_1, U_2$ of $x_1,x_2$ respectively and $V$ of the diagonal $\Delta=\{(x,x):x\in X\}$ such that $V\cap(U_1\times U_2)=\emptyset.$ Without loss of generality we may assume that $h_ix_0\in U_1$ for every $i\in I$ and $g_jx_0\in U_2$ for every $j\in J.$ For every $y\in X,$ there is a neighbourhood $U_y$ of $y$ such that $U_y\times U_y\subset V.$ Again by regularity, there are open $V_y$ for $y\in X$ such that $\overline{V_y}\subset U_y.$ By compactness, we can find $y_1,y_2,\ldots,y_n\in X$ such that $X=\bigcup_{i=1}^n \overline{V_{y_i}}.$

Let $\homeo(X)$ denote the group of homeomorphism of $X$ with the compact open topology. Since the action of $H$ on $X$ is continuous, the map $\phi: H\to \homeo(X)$ defined by $\phi(h)(x)=hx,$ is continuous. Since $\overline{V_{y_i}}\subset U_{y_i}$ for every $i,$ the set
$$
O=\bigcap_{i=1}^n \{f\in \homeo(X):f(\overline{V_{y_i}})\subset U_{y_i}\}
$$
is an open neighbourhoood of the indentity in $\homeo(X).$ Let $O_H=\phi^{-1}(O).$ Then $O_H$ is an open neighbourhood of the identity element in $H$ and whenever $h\in O_H,$ then $h\overline{V_{y_i}}\subset U_{y_i}$ for $i=1,2,\ldots,n,$ so $(y,hy)\in V$ for every $y\in X.$ Since the topology on $H$ is determined by finite substructures of $A,$ there exists $B\in\age(A)$ such that whenever $h_1,h_2\in H$ and $h_1|B=h_2|B$ then $h_1^{-1}h_2\in O_H.$ Since $\{(h_ix_0,h_i\prec):i\in I\}$ converges to $(x_1,<)$ and $\{(g_jx_0,g_j\prec):j\in J\}$ converges to $(x_2,<),$ there are $i_0\in I$ and $j_0\in J$ such that $h_{i_0}x_0\in U_1,$ $g_{j_0}x_0\in U_2$ and 
$$
h_{i_0}\prec|B=g_{j_0}\prec|B=<|B.
$$ 
That is to say, for every $b_1,b_2\in B,$
$$
h_{i_0}^{-1}(b_1)\prec h_{i_0}^{-1}(b_2) \leftrightarrow g_{j_0}^{-1}(b_1)\prec g_{j_0}^{-1}(b_2).
$$
If we denote $h^{-1}_{i_0}B=C$ and $g^{-1}_{j_0}B=D,$ then $(C,\prec|C)$ and $(D,\prec|D)$ are isomorphic via $\rho:h^{-1}_{i_0}(b)\mapsto g^{-1}_{j_0}(b)$ for $b\in B.$ Since $(A,\prec)$ is $\omega$-homogeneous, there exists $r\in H_{\prec}$ extending $\rho$ to all $(A,\prec).$ It means that $rh_{i_0}^{-1}(b)=g_{j_0}^{-1}(b)$ for every $b\in B,$ in other words $rh^{-1}_{i_0}|B=g^{-1}_{j_0}|B.$ By the choice of $B,$ $g_{j_0}r h^{-1}_{i_0}\in O_H.$ So $(h_{i_0}x_0,g_{j_0}rh^{-1}_{i_0}(h_{i_0}x_0))\in V.$ But $r\in H_{\prec},$ so $rx_0=x_0$ and therefore $(h_{i_0}x_0, g_{j_0}x_0)\in V.$ But also $(h_{i_0}x_0, g_{j_0}x_0)\in U_1\times U_2$,  which is a contradiction. 

$(b)$  Since $H_{\prec}$ is extremely amenable, every $H$-flow has a fixed point under the restricted action by $H_{\prec}.$ In particular, every minimal flow has a fixed point. Every point of a minimal flow has a dense orbit, so $\overline{H\prec}$ is universal among all minimal $H$-flows by part $(a)$. Since $\age(A,\prec)$ satisfies the ordering property, $\overline{H\prec}$ is a minimal $H$-flow by the previous theorem. Altogether we get that $\overline{H\prec}$ is the universal minimal flow of $H.$
\end{proof}

\begin{cor}
Let $A,\prec$ and $H$ be as above and let $G$ be a dense subgroup of $H.$ Then the universal minimal flow of $G$ is  $\overline{G\prec}.$
\end{cor}

\section{J\'{o}nsson structures}

In \cite{J1} and \cite{J2}, J\'onsson generalized Fra\"iss\'e's contruction (\cite{F}) of countable universal homogenous structures to uncountable structures. In general, the construction only works for relational structures of cardinality $\kappa$ with  $\kappa^{<\kappa}=\kappa.$ For instance, it is consistent both that there exists and that there does not exist a universal graph of cardinality $\aleph_1$ (see \cite{S2}). In the first article, \cite{J1}, J\'onsson was looking for conditions on a class $\K$ of relational structures that would give rise to a universal structure for $\K$. In the second article, \cite{J1}, he used the amalgamation property by Fra\"iss\'e to answer in positive a question of R. Baer whether the universal structure would be unique if an additional condition was imposed on the class $\K.$  Below, we recall J\'onsson's conditions (the third and fourth are the joint embedding property and the amalgamation property introduced by Fra\"iss\'e, the first and fourth are variations of the original conditions from \cite{J1} ensuring homogeneity in \cite{J2}, the last condition is a weakening of Fra\"iss\'e's condition that the class is hereditary).
\begin{itemize}
\item[I'.] For each ordinal $\xi$ there is $A\in\K$ of cardinality greater or equal to $\aleph_{\xi}.$
\item[II.] If $A\in \K$ and $A$ is isomorphic to a structure $B$ then $B\in\K.$
\item[III.] For every $A,B\in\K$ there exist $C\in\K$ and embeddings $f:A\to C$ and $g:B\to C.$
\item[IV'.] For every $A,B,C\in\K$ and embeddings $f:A\to B$ and $g:A\to C$ there exist $D\in\K$ and embeddings $i:B\to D$ and $j:C\to D$ such that $i\circ f=j\circ g.$ 
\item[V.] If $\l$ is a positive ordinal and  $\left< A_{\xi}:\xi\in\l \right>$ is a sequence of structures in $\K$ such that $A_{\xi}\leq A_{\nu}$ whenever $\xi<\nu<\l,$ then $\bigcup_{\xi<\l}A_{\xi}\in \K.$
\item[VI$_{\alpha}$.] If $A\in\K,$ $B\leq A$ and the cardinality of $B$ is less than $\aleph_{\alpha},$ then there exists $C\in\K$ such that $B\leq C\leq A$ and the cardinality  of $C$ is less than $\aleph_{\alpha}.$
\end{itemize}

We will call a class of relational structures satifying the conditions  I'., II, III, VI'., V. and VI$_{\alpha}.,$ for some positive ordinal $\alpha$ a \emph{J\'onsson class} (for $\alpha$).

Let us make precise what we mean by a universal and a homogeneous structure.

\begin{definition}[Universality; \cite{J1}]
Let $\K$ be a class of relational structures and let $\alpha$ be an ordinal. We say that a structure $A\in\K$ is \emph{$(\aleph_{\alpha},\K)$-univeral} if $A$ has cardinality $\aleph_{\alpha}$ and every $B\in \K$ of cardinality $\leq\aleph_{\alpha}$ can be embedded into $A.$
\end{definition}

\begin{definition}[Homogeneity; \cite{J2}]
Let $\K$ be a class of relational structures and let $\alpha$ be an ordinal. We say that a structure $A$ is \emph{($\aleph_{\alpha}, \K$)-homogeneous} if $A\in\K,$ the cardinality of $A$ is $\aleph_{\alpha}$ and for any substructure $B\in \K$ of $A$ of a  smaller cardinality, every embedding of $B$ into $A$ can be extended to an automorphism of $A.$
\end{definition}

The following proposition gives us a tool to check homogeneity of structures. The proof goes by a back-and-forth argument.

\begin{prop}[Extension property] \label{ep}
Let $\K$ be a class of structures and $A$ a structure of cardinality $\aleph_{\alpha}$ for some ordinal $\alpha.$ Then $A$ is $(\aleph_{\alpha}, \K)$-homogeneous if and only if $A\in\K$ and for every $B,C\in\K$ of cardinalities less than $\aleph_{\alpha}$ and embeddings $i:B\to A$ and $j:B\to C,$ there exists and embedding $k:C\to A$ such that $i=k\circ j.$
\end{prop}

The main result of  \cite{J1} and \cite{J2} is the existence of a universal homogeneous structure for a J\'onsson class of cardinality $\aleph_{\alpha}$ whenever $\aleph_{\alpha}^{<\aleph_{\alpha}}=\aleph_{\alpha}.$

\begin{theorem}[\cite{J2}]
Let $\alpha$ be a positive ordinal with the following two properties:
\begin{itemize}
\item[(i)] If $\l<\omega_{\alpha}$ and if $n_{\nu}<\aleph_{\alpha}$ whenever $\nu<\l,$ then $\sum_{\nu<\l} n_{\nu}<\aleph_{\alpha}.$
\item[(ii)] If $n<\aleph_{\alpha},$ then $2^{n}\leq \aleph_{\alpha}.$
\end{itemize}
If $\K$ is a J\'onsson class for $\alpha,$ then there exists a unique $(\aleph_{\alpha},\K)$-universal homogeneous structure.

\end{theorem}

Nowadays, we abbreviate the conditions (i) and (ii) in the previous theorem as $$\aleph_{\alpha}^{<\aleph_{\alpha}}=\aleph_{\alpha}$$ and say that $\aleph_{\alpha}$ to the weak power $\aleph_{\alpha}$ is equal to $\aleph_{\alpha}.$

\begin{cor}[\cite{J2}]
If the General Continuum Hypothesis holds and $\K$ is a J\'onsson class for every positive ordinal $\alpha$, then there is an $(\aleph_{\alpha},\K)$-universal homogeneous structure for every positive $\alpha.$
\end{cor}

In the previous section, we presented a method which shows how to compute universal minimal flows of groups of automorphisms of structures provided that the structures admit a certain linearly ordered extension. To ensure that for J\'onsson structures, we require the corresponding J\'onsson class to permit an extension to a J\'onsson class of linearly ordered structures satisfying a mild condition of \emph{reasonability} as in the case of Fra\"iss\'e classes in \cite{KPT}. 

\begin{definition}[Reasonable class]
Let $L\supset \{>\}$ be a language. Let $\K_<$ be a class of $L$-structures with $<$ interpreted as a linear ordering. Let $L_0=L\setminus \{<\}$ and $\K=\K_<|L_0.$ We say that $K_<$ is \emph{reasonable}, if whenever $A,B\in\K,$ $A\leq B$ and $\prec$ is a linear ordering on $A$ with $(A,\prec)\in\K_<,$ there exists a linear ordering $\prec'$ on $B$ such that $(B,\prec')\in\K_<$ and $(A,\prec)\leq (B,\prec').$
\end{definition}

The following proposition is Proposition 5.2 from \cite{KPT} adjusted to J\'onsson classes. It verifies that if we remove the linear order from the universal homogeneous structure for a reasonable J\'onsson class, we obtain the universal homogeneous structure for the corresponding class of structures without linear orderings.

\begin{prop}
Suppose that $\alpha$ is a positive ordinal such that $\aleph_{\alpha}^{<\aleph_{\alpha}}=\aleph_{\alpha}.$ Let $\K_<$ be a J\'onsson class for $\alpha$  in a language $L\supset \{<\}$ with $<$ interpreted as a linear ordering. Suppose that $\K_<$ is closed under substructures. Let $(A,<)$  be the $(\aleph_{\alpha},\K_<)$-universal homogeneous structure. Set $L_0=L\setminus \{<\}$ and $ \K=\K_<|L_0.$   Then the following are equivalent:
\begin{itemize}
\item[(a)] $\K_<$ is reasonable,
\item[(b)] $\K$ is a J\'onsson class and $A=(A,<)|L_0$ is an $(\aleph_{\alpha},\K)$-universal homogeneous structure
\end{itemize}
\end{prop}

\begin{proof}
$(\Rightarrow)$
Obviously, $\K$ satisfies conditions I.',II.,V. and IV$_{\alpha}.$ To verify the joint embedding property III., let $B,C\in\K.$ Find $\prec,\prec'$ linear orderings on $B,C$ respectively such that $(B,\prec),(C,\prec')\in\K_<.$ Since $\K_<$ satisfies III., there is $D_<\in\K_<$ such that $(B,\prec),(C,\prec')\leq D_<.$ Then $D=D_<|L_0$ is a witness of joint embedding of $B$ and $C$ in $\K.$

It remains to show the amalgamation property IV'.: Fix $B,C,D\in\K$ and embeddings $i:B\to C$ and $j:B\to D.$ Let $\prec$ be an ordering on $B$ such that $(B,\prec)\in\K_<.$ Since $\K_<$ is reasonable, we can find linear orders $\prec',\prec''$ on $C, D$ respectively such that $(C,\prec'),(D,\prec'' )\in\K_<$ and $i:(B,\prec)\to (C,\prec'), j:(B,\prec)\to (D,\prec'')$ are still embeddings. Amalganation property for $\K_<$ provides us with $E_<\in\K_<$ and embeddings $k:(C,\prec')\to E_<, l:(d,\prec'')\to E_<$ such that $k\circ i=l\circ j.$ Let $E=E_<|L_0$ and embeddings $k:C\to E, l:D\to E$ are witnesses of amalgamation of $C$ and $D$ over $B$ in $\K.$ 

Finally, we check that $A$ is $(\aleph_{\alpha},\K)$-universal and homogeneous. Let $B\in\K$ of cardinality $\leq \aleph_{\alpha}$ and let  $\prec$ be a linear ordering on $B$ such that $(B,\prec)\in\K_<.$ Since $(A,<)$ is $(\aleph_{\alpha},K_<)$-universal,  there is an embedding $i:(B,\prec)\to (A,<)$ which is also an embedding from $B$ to $A.$ It means that $A$ is $(\aleph_{\alpha},\K)$-universal. 

To show that $A$ is $(\aleph_{\alpha},\K)$-homogeneous, it is enough to check the extension property in Proposition \ref{ep}. For that, let $B\leq C$ be structures in $\K$ with cardinality $<\aleph_{\alpha}$ and let $i:B\to A$ and  $j:B\to C$ be embeddings. Denote by $\prec=i^{-1}(<|i(B)).$ Then $(B,\prec)\in\K_<.$ Since $\K_<$ is reasonable, there exists $\prec'$  on $C$ with $(C,\prec')\in\K_<$ such that $j:(B,\prec)\to (C,\prec')$ is also an embedding. Since $(A,<)$ is $(\aleph_{\alpha},\K_<)$-homogeneous, it satisfies the extension property, i.e. there is an embedding $k:(C,\prec')\to (A,<)$ such that $i=k\circ j$ and we are done.  
 
 $(\Leftarrow)$ Fix $B,C\in\K$ and an embedding $i:B\to C.$ Let $\prec$ be a linear ordering on $B$ such that $(B,\prec)\in\K_<.$ Then there is an embedding $j:(B,\prec)\to (A,<),$ which is of course also an embedding from $B$ to $A$. Since $A$ is homogeneous, it satisfies the extension property in Proposition \ref{ep}., so there is an embedding $k:C\to A$ extending $j.$ Let $\prec'=j^{-1}(<|j(C)).$ Then $(C,\prec')\in \K_<$ and $i:(B,\prec)\to (C,\prec')$ is an embedding. 
\end{proof}

We are ready to apply Theorem \ref{T2} and Theorem \ref{TUM} to J\'{o}nsson structures.

\begin{theorem}\label{TJ}
Let $L\supset \{<\}$ be a relational signature and let $\alpha$ be a positive ordinal such that $\aleph_{\alpha}^{<\aleph_{\alpha}}=\aleph_{\alpha}.$ Let $\K_<$ be a reasonable J\'onsson class in the signature $L$ with $<$ interpreted as a linear order and let $\K_<$ be closed under substructures. Denote by $(A,<)$ the ($\aleph_{\alpha},\K_<$)-universal homogeneous structure. 
\begin{itemize}
\item[(a)] If $\age(A,<)$ satisfies the Ramsey property and $G_<$ is a dense subgroup of $\aut(A,<),$ then $G_<$ is extremely amenable. 
\item[(b)] If $\age(A,<)$ satisfies both the Ramsey property and the ordering property, then the universal minimal flow of any dense subgroup $G$ of $\aut(A)$ is $\overline{G<}.$
\end{itemize}
\end{theorem}

Let us turn to concrete examples of J\'onsson classes to which we can apply the above theorem.

\begin{prop} \label{PS}
Let $\K$ be one of the following classes:
\begin{itemize}
\item graphs or graphs with arbitrary linear orderings,
\item $K_n$-free graphs or $K_n$-free graphs with arbitrary linear orderings, 
\item hypergraphs or hypergraphs with arbitrary linear orderings,
\item $\mathcal{A}$-free hypergraphs or $\mathcal{A}$-free hypergraphs  with arbitrary linear orderings,
\item posets or posets with linear orderings extending the partial order.
\end{itemize}
Then $\K$ is a J\'onsson class for every positive ordinal $\alpha$ and it is closed under substructures.
\end{prop}

\begin{proof}
It is easy to see that in all cases, $\K$ satisfies conditions I'.,II., V. and that it is closed under substructures which is a strengthening of  VI$_{\alpha}.$ for every $\alpha.$ We show that they also satisfy III. and IV'.:

To satisfy the joint embedding property III. for structures $A,B\in\K,$ we take $C$ to be the disjoint union of $A$ and $B$ and the embeddings to be the identity. If $A,B$ also possess a linear order, then let elements of $A$ precede elements of $B$ in $C.$ 

The amalgamation property IV'. is proved similarly as in the case of finite structures: Let $A,B,C\in\K$ and let $i:A\to B, j:A\to C$ be embeddings. Let the underlying set of $D$ be a quotient of the disjoint union of $B$ and $C$ via an equivalence relation $\sim,$ where $b\sim c$ if and only if $b\in B, c\in C$ and there is an $a\in A$ such that $i(a)=b, j(a)=c.$ Let $k:B\to D$ and $l:C\to D$ be the identity injections. Then obviously, $k\circ i=l\circ j.$  Now we equip $D$ with a structure of the correct type to make sure that $k$ and $l$ are embeddings:
\begin{itemize}
\item If $\K$ is a class off hypergraphs in a signature $L$ and $R_i\in L$ has arity $n,$ then $$(d_1, d_2, \ldots, d_{n})\in R^D_{i}\subset D^n$$ if and only if either $$(k^{-1}(d_1),k^{-1}(d_2),\ldots,k^{-1}(d_{n}))\in R^B_{i}\subset B^n$$ or $$(l^{-1}(d_1),l^{-1}(d_2),\ldots, l^{-1}(d_{n}))\in R^C_{i}\subset C^n.$$ It means that $R_i^D=R_i^B\cup R_i^C$ when we identify $B,C$ with their corresponding images $k(B),l(C).$ Notice that $D$ will be a graph ($K_n$-free graph, $\mathcal{A}$-free hypergraph) if $A, B, C$ are.
\item If $\K$ is the class of posets with $\prec$ the symbol for the partial order, then  $\prec^D$ is the transitive closure of $(k\times k)(\prec^B)\cup (l\times l)(\prec^C).$
\end{itemize}

If $\K$ is a class of linearly ordered structures with $<$ the symbol for the linear order, then $<^D$ on $D$ is an arbitrary linear ordering extending the transitive closure of $(k\times k)(<^B)\cup (l\times l)(<^C).$ In the next section, we will need to be more careful when amalgamating linear orders and set elements of $B$ precede elements of $C$ whenever we can. Formally, for $b\in B$ denote $(-,b)=\{a\in A:k(a)<_B b\}$ and for $c\in C$ denote $(-,c)=\{a\in A: l(a)<_C c\}.$ Let $<_D$ be the extension of both $k\times k (<_B)$ and $l\times l (<_C)$ such that whenever $b\in B, c\in C$ then $k(b)<_D l(c)$ if and only if $(-,b)\subseteq (-,c).$ It is easy to see that $<^D$ extends $\prec^D$ if $\K$ is the class of posets with a linear ordering extending the partial order. So in all cases, $\K$ is a J\'onsson class for every $\alpha.$
\end{proof}

We know that if $\K$ is one of the ordered classes in the proposition above, then finite structures in $\K$ satisfy the ordering property and the Ramsey property. Also, the ordered classes are obviously reasonable. Therefore, Theorem \ref{TJ} applies to $\K$.

\begin{cor}
Let $\kappa$ be a cardinal satisfying $\kappa^{<\kappa}=\kappa$ and let $(A,<)$ be the universal homogeneous ordered graph ($K_n$-free graph, hypergraph, $\mathcal{A}$-free hyperhraph, poset) of cardinality $\kappa.$ Let $G_<$ be a dense subgroup of $\aut(A,<)$ and let $G$ be a dense subroup of $\aut(A).$ Then $G_<$ is extremely amenable and the universal minimal flow of $G$ is $\overline{G<}.$
\end{cor}

\section{Constructing $\omega$-homogeneous structures in every cardinality}

In this section, we overcome the restriction on the size of structures given by J\'{o}nsson's contruction for the price of losing universality and only keeping $\omega$-homogeneity. This is however sufficient to compute universal minimal flows of groups of automorphisms of such structures and provides many new examples. 

We present a construction of $\omega$-homogeneous structures considered in the Proposition \ref{PS} in every uncountable cardinality:

Let $A$ be a graph ($K_n$-free graph, hypergraph,  $\mathcal{A}$-free hypergraph, poset) of cardinality $\kappa$ and let $<$ be an arbitrary linear ordering on $A$ (respectively an ordering extending the partial order if $A$ is a poset). We will construct an $\omega$-homogeneous structure $(A_{\kappa},<_{\kappa})$ of cardinality $\kappa$ in which $(A,<)$ is embedded and such that $A_{\kappa}$ itself is $\omega$-homogeneous.

By induction, we construct a chain of superstructures $\left<(A_{\l},<_{\l}): \l<\k\right>$ of $(A,<)$ such that $(A_{\l},<_{\l})\leq (A_{\mu},<_{\mu})$ whenever $\l<\mu<\k$ and $(A_{\kappa},<_{\kappa})=\bigcup_{\l<\kappa} (A_{\l},<_{\l}).$ In step $\l,$ we deal with a pair of isomorphic finite substructures $(F_{\l},<_{\l})$ and $(G_{\l},<_{\l})$ of $(A_{\l},<_{\l})$ and an isomorphism $\phi_{\l}:F_{\l}\to G_{\l}$ and we construct $(A_{\l+1},<_{\l+1})$ with an automorphism $\psi_{\l+1}:A_{\l+1}\to A_{\l+1}$ extending $\phi_{\l}$ that is order-preserving if $\phi_{\l}$ is. Moreover, we make sure that if $\xi<\l$ and $(F_{\xi},<_{\xi})=(F_{\l},<_{\l})$ and $(G_{\xi},<_{\xi})=(G_{\l},<_{\l})$ and $\phi_{\xi}=\phi_{\l},$ then $\psi_{\xi}\subset\psi_{\l}.$

To ensure that $A_{\kappa}$ and $(A_{\kappa},<_{\kappa})$ are both $\omega$-homogeneous, we need to consider every triple $((F_{\l},<_{\l}),(G_{\l},<_{\l}),\phi_{\l}:F_{\l}\to G_{\l})$ cofinality of $\kappa$-many times for every $\l<\k$. In order to do that, we fix a book-keeping function - a bijection $f:\kappa\to \kappa \times \kappa$ - which we only need to satisfy that whenever $f(\l)=(\mu,\xi),$ then $\mu\leq \l.$ We will describe $f$ in detail later. 

If $\l$ is limit, then $(A_{\l},<_{\l})=\bigcup_{\mu<\l} (A_{\mu},<_{\mu}).$ If $\l=\mu +1,$ then we construct $(A_{\l},<_{\l})$ from $(A_{\mu},<_{\mu})$ as follows. 

Let $(F_{\mu},<_{\mu}),$ $(G_{\mu},<_{\mu})$ be the pair of finite substructures of $(A_{\mu},<_{\mu})$ and $\phi_{\mu}:F_{\mu}\to G_{\mu}$ the automorphism given by $f$ in step $\mu$ (as described later). Denote by $(B,<_B)$ the union of all $(A_{\xi},<_{\xi})$ such that $\xi\leq \mu$ and $(F_{\xi},<_{\xi}) = (F_{\mu},<_{\mu}),$ $(G_{\xi},<_{\xi}) = (G_{\mu},<_{\mu})$  and $\phi_{\xi} = \phi_{\mu}$ and let $\psi_B$ be the union of the corresponding $\psi_{\xi}$'s. We need to set $\psi_{\l}|B=\psi_B.$ 
For every $a\in A_{\mu}\setminus B$ we add two new points $a^1$ and $a^{-1}$ to be the image and the preimage of $a$ under $\psi_{\l}$ respectively. In the next step, we need to add  for every $a\in A_{\mu}\setminus B$ another two points $a^2, a^{-2}$ to be the image of $a^1$ and the preimage of $a^{-1}$ under $\psi_{\l}$ respectively. We  continue in the same manner $\omega$-many times until we have taken care that every  point has its image and preimage. Formally, let  $A_{\mu}^z$ be a copy of $A_{\mu}$ for $z\in\mathbb{Z}$ and denote by $a^z$ the element  of $A_{\mu}^z$ corresponding to an $a$ in $A_{\mu}.$ We identify $A_{\mu}^0$ with $A_{\mu.}$ 

We need to distinguish two cases, depending on whether  the triple $((F_{\mu},<_{\mu}), (G_{\mu},<_{\mu}),\phi_{\mu})$ appears for the first time throughout the construction (i.e. $B=\emptyset$) or not. However, in both cases the construction follows a similar pattern.

Case 1: $B=\emptyset.$ Let $F_{\mu}^z$ and $G_{\mu}^z$ denote the corresponding copies of $F_{\mu}$ and $G_{\mu}$ in $A_{\mu}^z$ respectively. The underlying set of $A_{\l}$ will be a quotient of the disjoint union of $A_{\mu}^z$ for $z\in\mathbb{Z}$ by an equivalence relation $\sim$ gluing $G_{\mu}^z$ with $F_{\mu}^{z+1}$ via $\phi_{\mu}:$
$$
a_1^{z_1}\in F_{\mu}^{z_1}, a_2^{z_2}\in G_{\mu}^{z_2} \rightarrow (a_1^{z_1}\sim a_2^{z_2} \leftrightarrow (\phi_{\mu}(a_1)=a_2 \wedge z_1=z_2+1)).
$$

If $a^{z}\in A_{\mu}^z$ and $a\notin F_{\mu}$ then we write $a^{z}$ to mean its  $\sim$-class $[a^{z}]_{\sim} = \{a^z\}.$ If $a\in F_{\mu},$ then we write $a^z$ to mean its $\sim$-class $[a^z]_{\sim}=\{a^z,\phi_{\mu}(a)^{z-1}\}.$

Let us define $\psi_{\l}:A_{\l}\to A_{\l}$ to be the mapping $a^z\mapsto a^{z+1}$ whenever $a\notin F_{\mu}$ and $z\in\mathbb{Z}$ and $a^z\mapsto \phi_{\mu}(a)^z$ whenever $a\in F_{\mu}.$

If $\phi_{\mu}$ is order preserving, let us transfer the linear order $<_{\mu}$ on $A_{\mu}$ to $A_{\mu}^z$ for every $z\in\mathbb{Z}:$

$$
a^z<_{\mu}^z b^z \mbox{ if \ and \ only \ if } \psi_{\l}^{-z}(a^z)<_{\mu}\psi_{\l}^{-z}(b^z)
$$

We define the structure on $A_{\l}$ inductively using amalgamation as described in the previous section. If $\phi_{\mu}$ is order preserving, then we use amalgamation for linearly ordered structures letting elements of $A_{\mu}^{z_1}$ precede elements of $A_{\mu}^{z_2}$ for $z_1<z_2$ whenever we can as described in the previous section. Otherwise, we use amalgamation for unordered structures and in the end obtain $<_{\l}$ as an arbitrary linear extension of $<_{\mu}$ (and extending the partial order on $A_{\l}$ if $A$ is a poset).

Let $n\in \omega$ and let $A_{\l}^{n-}$ denote the subset of $A_{\l}$ consisting of points in $A_{\mu}^z$ for $z\in\{-n+1,-n+2,\ldots, n\},$ i.e.
$$A_{\l}^{n-}=(\bigcup_{z=-n+1}^n A_{\mu}^z)/\sim.$$
Similarly let $A_{\l}^{n+}$ denote the quotient  
$$A_{\l}^{n+}=\bigcup_{z=-n}^n A_{\mu}^z/\sim$$ 
of the disjoint union of $A_{\mu}^z$ for $z\in\{-n,-n+1,\ldots,n\}$ by $\sim.$
We have that $A_{\l}^0=A_{\mu}.$  Suppose that the structure on $A_{\l}^{n+}$ has been defined. Then the structure on $A_{\l}^{(n+1)-}$ is given by amalgamating $A_{\l}^{n+}$ with $A_{\mu}^{n+1}$ along 
$$F^{n+1}_{\mu}\to A_{\l}^{n+}, a^{n+1}\mapsto (\phi_{\mu}(a))^n \mbox{ and } F^{n+1}_{\mu} \to A_{\mu}^{n+1}, a^{n+1}\mapsto a^{n+1}.$$ When the structure on $A_{\l}^{n-}$ has been defined, then the structure on $A_{\l}^{n+}$ is given by amalgamation of $A_{\l}^{n-}$ and $A_{\mu}^{-n}$ along 
$$F^{-n+1}_{\mu}\to A_{\l}^{n-}, a^{-n+1}\mapsto a^{-n+1} \mbox{ and } F^{-n+1}_{\mu} \to A_{\mu}^{-n}, a^{-n+1}\mapsto (\phi_{\mu}(a))^{-n}.$$

If $\phi_{\mu}$ is order preserving, then $A_{\l}^{n+}$ (respectively $A_{\mu}^{-n}$) plays the role of $B$ and $A_{\mu}^{n+1}$ (respectively $A_{\l}^{n-}$) the role of $C$ in the amalgmation of linear orders described in the previous section.

We can see that $A_{\l}^{n+}\leq A_{\l}^{n+1-}\leq A_{\l}^{n+1+}$ for every $n\in\omega,$ so we can define the structure on $A_{\l}$ as the union of the chain $\left< A_{\l}^{n+}:n\in\omega \right>$ (equivalently $\left< A_{\l}^{n-}:n\in\omega \right>$):

$$
A_{\l}=\bigcup_{n\in\omega} A_{\l}^{n+}=\bigcup_{n\in\omega} A_{\l}^{n-}.
$$



Case 2: $B\neq\emptyset.$ The underlying set of $A_{\l}$ will then be a quotient of the disjoint union of $A_{\mu}^z$ for $z\in\mathbb{Z}$ by an equivalence relation $\sim$ gluing corresponding copies of $B$ in $A_{\mu}^z$ via $\psi_B:$
$$a_1^{z_1}\in A_{\mu}^{z_1},  a_2^{z_2}\in A_{\mu}^{z_2} \rightarrow (a_1^{z_1}\sim a_2^{z_2} \longleftrightarrow a_1,a_2\in B \wedge \psi^{-z_1}_B (a_1)=\psi^{-z_2}_B (a_2)),$$
We identify $\psi_B^0$ with $\psi_B.$

If $a^{z}\in A_{\mu}^z$ and $a\notin B$ then we write $a^{z}$ to mean its  $\sim$-class $[a^{z}]_{\sim} = \{a^z\}.$ If $a\in B,$ then we write $a$ to mean its $\sim$-class $[a]_{\sim}=\{a^z:z\in\mathbb{Z}\}.$

Let us define $\psi_{\l}:A_{\l}\to A_{\l}$ to be the mapping $a^z\mapsto a^{z+1}$ whenever $a\notin B$ and $z\in\mathbb{Z}$ and $a\mapsto \psi_B(a)$ whenever $a\in B.$

As in Case 1, we can obtain the structure on $A_{\l}$ inductively by amalgamation along $B$ in $\omega$-many steps. However, we can obtain $A_{\l}$ via amalgamation of $A_{\mu}^z, z\in\mathbb{Z}$ along $B\to A_{\mu}^z, b\mapsto \psi_{\l}^z(b)$ in one step:
\begin{itemize}
\item If $A$ is a hypergraph in a signature $L$ and $R_i\in L$ has arity $n,$ then $$(a_1^{z_1},a_2^{z_2},\ldots, a_{n}^{z_n})\in R_{i}^{A_{\l}}$$ if and only if $$(z_1=z_2=\ldots=z_n)
 \wedge (\psi_{\l}^{-z_1}(a_1^{z_1}),\psi_{\l}^{-z_1}(a_2^{z_1}),\ldots,\psi_{\l}^{-z_1}(a_n^{z_1}))\in R_{i}^{A_{\mu}}.$$ Notice that $A_{\l}$ will be a graph ($K_n$-free graph, $\mathcal{A}$-free hypergraph) if $A_{\mu}$ is.
\item If $A$ is a poset with a partial order $\prec$, then  $\prec_{\l}$ is the transitive closure of the relation obtained as in the case of hypergraphs.
\end{itemize}
 
It remains to extend the linear order $<_{\mu}$ to a linear order $<_{\l}$ on $A_{\l}.$ If $\phi_{\mu}$ is not order preserving, let $<_{\l}$ be an arbitrary linear ordering extending $<_{\mu}$ (and extending the partial order $\prec_{\l}$ on $A_{\l}$ in case that $A$ is a poset). If $\phi_{\mu}$ is order preserving, we will extend $<_{\mu}$ to $<_{\l}$ to make sure that $\psi_{\l}$ is order preserving.


We can also describe $<_{\l}$ explicitly, without an inductive construction. We again let elements in $A_{\mu}^{z_1}$ precede elements in $A_{\mu}^{z_2}$ whenever we $z_1<z_2$ and we are free to do so: 
Let $a^z\in A_{\mu}^z$ and denote by
$$(-,a^z)=\{b\in B: b<_{\mu}\psi_{\l}^{-z}(a)\}.$$ 
Let  $a_1^{z_1}\in A_{\mu}^z$ and $a_2^{z_2}\in A_{\mu}^{z_2}$ for some $z_1,z_2\in\mathbb{Z}$ and $a_1^{z_1}\neq a_2^{z_2}.$ 
\begin{itemize}
\item If $z_1=z_2,$ then $a_1^{z_1}<_{\l}a_2^{z_2}$ if and only if $\psi_{\l}^{-z_1}(a_1^{z_1})<_{\mu}\psi_{\l}^{-z_1}(a_2^{z_2}).$
\item If $z_1\neq z_2,$ then $a_1^{z_1}<_{\l}a_2^{z_2}$ if and only if $(-,a_1^{z_1}) \subsetneq (-,a_2^{z_2})$ or $(-,a_1^{z_1})=(-,a_2^{z_2})$ and $z_1<z_2.$
\end{itemize}

Obviously, $<_{\l}$ is antireflexive and antisymmetric. Let us check transitivity: Let $a_1^{z_1},a_2^{z_2},a_3^{z_3}\in A_{\l}$ and $a_1^{z_1}<a_2^{z_2}$ and $a_2^{z_2}<a_3^{z_3}.$ We have five possible cases:
\begin{itemize}
\item[1.] $z_1=z_2=z_3.$ Then $\psi_{\l}^{-z_1}(a_1^{z_1})<_{\mu}\psi_{\l}^{-z_1}(a_2^{z_1})<_{\mu}\psi_{\l}^{-z_1}(a_3^{z_1})$ so $\psi_{\l}^{-z_2}(a_1^{z_1})<_{\mu} \psi_{\l}^{-z_1}(a_3^{z_1})$ by transitivity of $<_{\mu}$.
\item[2.] $z_1=z_2\neq z_3.$ Then $\psi_{\l}^{-z_1}(a_1^{z_1})<_{\mu}\psi_{\l}^{-z_1}(a_2^{z_1}),$ so $(-,a_1^{z_1})\subseteq (-,a_2^{z_1}),$ and either $(-,a_2^{z_1}) \subsetneq (-,a_3^{z_3})$ or $(-,a_2^{z_1}) = (-,a_3^{z_3})$ and $z_2<z_3.$
\item[3.] $z_1\neq z_2=z_3.$ Then $(-,a_1^{z_1}) \subsetneq (-,a_2^{z_2})$ or $ (-,a_1^{z_1}) = (-,a_2^{z_2})$ with $z_1<z_2$ and $\psi_{\l}^{-z_2}(a_2^{z_2})<_{\mu}\psi_{\l}^{-z_2}(a_3^{z_2}),$ so $(-,a_2^{z_2})\subseteq (-,a_3^{z_2}).$
\item[4.] $z_1=z_3\neq z_2.$ Then $(-,a_1^{z_1})\subseteq (-,a_2^{z_2}) \subseteq (-,a_3^{z_1})$ and at least one inclusion is proper, since otherwise $z_1<z_2<z_1.$ So we get $(-,a_1^{z_1})\subsetneq (-,a_3^{z_1}),$ which implies $\psi_{\l}^{-z_1}(a_1^{z_1})<_{\mu}\psi_{\l}^{-z_1}(a_3^{z_1})$.
\item[5.] $z_1\neq z_2\neq z_3, z_1\neq z_3.$ Then we have the following four cases:

\begin{itemize}
\item[(i)] $(-,a_1^{z_1}) \subsetneq(-,a_2^{z_2}) \subsetneq (-,a_3^{z_3}).$
\item[(ii)] $(-,a_1^{z_1}) \subsetneq(-,a_2^{z_2})=(-,a_3^{z_3})$ and $z_2<z_3.$ 
\item[(iii)] $(-,a_1^{z_1})=(-,a_2^{z_2}) \subsetneq (-,a_3^{z_3})$ and $z_1<z_2$
\item[(iv)] $(-,a_1^{z_1})=(-,a_2^{z_2})=(-,a_3^{z_3})$ and $z_1<z_2<z_3.$ 
\end{itemize}
\end{itemize}
In all cases we get that  
$a_1^{z_1}<_{\l}a_3^{z_3},$ hence $<_{\l}$ is transitive.

If $A$ is a poset with $\prec$ the symbol for the partial order and $\prec_{\l}$ is the extension of $\prec$ to $A_{\l}$ described above, let us verify that $<_{\l}$ extends $\prec_{\l}:$ If $a_1^z,a_2^z\in A_{\mu}^z$ for some $z\in\mathbb{Z},$ then $a_1^{z}<_{\l}a_2^{z}$ if and only if $a_1^{z}\prec_{\l}a_2^{z}$ by definition. If $a_1^{z_1}\in A_{\mu}^{z_1},a_2^{z_2}\in A_{\mu}^{z_2}$ for some $z_1\neq z_2$ and $a_1^{z_1}\prec_{\l}a_2^{z_2},$ then there is $a\in B$ such that $\psi_{\l}^{-z_1} (a_1^{z_1})<_{\mu} a$ and $a<_{\mu} \psi_{\l}^{-z_2}(a_2^{z_2}).$ It means that $a\in (-,a_1^{z_1})$ while $a\notin (-,a_2^{z_2}),$ showing that $a_1^{z_1}<_{\l}a_2^{z_2}.$

Let us go back to the definition of the book-keeping function $f$: When $(A_{\mu},<_{\mu})$ has been constructed, let $\{((F_{\mu}^{\xi},<_{\mu}^{\xi}),(G_{\mu}^{\xi},<_{\mu}^{\xi}),\phi_{\mu}^{\xi}):\xi<\k\}$ be an enumeration of all triples such that $(F_{\mu}^{\xi},<_{\mu}^{\xi})$ and $(G_{\mu}^{\xi},<_{\mu}^{\xi})$ are finite substructures of $(A_{\mu},<_{\mu})$ and $\phi_{\mu}^{\xi}$ is an isomorphism between $F_{\mu}^{\xi}$ and $G_{\mu}^{\xi}$ (this is possible since $\k^{<\omega}=\k$ for every infinite $\k$). If $f(\mu)=(\nu,\xi),$ then the triple we consider in step $\mu$ is $((F_{\nu}^{\xi},<_{\nu}^{\xi}),(G_{\nu}^{\xi},<_{\nu}^{\xi}),\phi_{\nu}^{\xi}).$ The inequality $\nu\leq\mu$ ensures that  $F_{\nu}^{\xi}$ and $G_{\nu}^{\xi}$ are finite substructures of already defined $A_{\nu}\leq A_{\mu}.$ The inclusion $A_{\nu}\subset A_{\mu}$ for $\nu<\mu<\k$ provides that for every $\mu<\k,$ every triple $((F,<_{\mu}),(G_{\mu}),\phi)$ of two finite substructures of $(A_{\mu},<_{\mu})$ and an isomorphism between $F$ and $G$ will be considered unboundedly many times throughout our construction.

\bigskip

The outcome of the construction can be formulated as the following theorem.

\begin{theorem}
Let $A$ be a graph ($K_n$-free graph, hypergraph, $\mathcal{A}$-free hypergraph, poset) of an infinite cardinality $\kappa$ and let $<$ be an arbitrary linear ordering on $A$ (respectively an ordering extending the partial order if $A$ is a poset).  Then there exists a linearly ordered graph ($K_n$-free graph, hypergraph, $\mathcal{A}$-free hyperhraph, poset) $(A',<')$ of cardinality $\kappa$ in which $(A,<)$ is embedded and such that both  $A'$ and $(A',<')$ are $\omega$-homogeneous (and if $A$ is a poset, $<'$ extends the partial order on $A'$).
\end{theorem}

If we start with $A$  a graph ($K_n$-free graph, hypergraph, $\mathcal{A}$-free hyperhraph, poset) and a linear order $<$ on $A$ such that $\age((A,<))$ is the class of all finite linearly ordered graphs ($K_n$-free graphs, hypergraphs, $\mathcal{A}$-free hyperhraphs) or posets with the linear order extending the partial order, the construction provides us with a structure for which we are able to compute the universal minimal flow of its group of automorphisms. This can easily be arranged for instance by requiring that $(A,<)$ contains a copy of the countable $\omega$-homogeneous ordered graph ($K_n$-free graph, hypergraph, $\mathcal{A}$-free hypergraph, poset).

\begin{theorem}
Let $A$ be an $\omega$-homogeneous graph ($K_n$-free graph, hypergraph, $\mathcal{A}$-free hypergraph, poset) and let $<$ be a linear ordering on $A$ (extending the partial order if $A$ is a poset) such that $(A,<)$ is $\omega$-homogeneous as well. 
Suppose that $\age(A,<)$ is the class of all finite linearly ordered graphs ($K_n$-free graphs, hypergraphs, $\mathcal{A}$-free hyperhraphs) or posets with linear orderings extending the partial order. 

If $G$ is a dense subgroup of $\aut(A),$ then the universal minimal flow of $G$ is the space of all linear orderings on $A$ (respectively the space of all linear orderings extending the partial order if $A$ is a poset).
\end{theorem}

\section*{Acknowledgement}
I wish to thank  Stevo Todor\v{c}evi\'{c} for proposing the project of computing universal minimal flows of groups of automorphisms of uncountable structures, for introducing me to J\'onsson structures and suggesting the present construction of $\omega$-homogeneous structures. I am also grateful to Lionel Nguyen Van Th\'e for careful reading and valuable suggestions.

\bibliographystyle{alpha}
\bibliography{Linear}

\end{document}